\documentclass[a4paper]{amsart}

\usepackage{amsfonts,amsmath,amssymb,amsthm}
\usepackage{cite}
\usepackage{enumerate}
\usepackage[utf8]{inputenc}

\newtheorem{lemma}{Lemma}[section]
\newtheorem{theorem}[lemma]{Theorem}
\newtheorem{corollary}[lemma]{Corollary}
\theoremstyle{remark}
\newtheorem{remark}[lemma]{Remark}
\newtheorem{problem}[lemma]{Problem}

\newcommand{\N}{\mathbb N}

\newcommand{\R}{\mathbb R}
\newcommand{\Z}{\mathbb Z}

\newcommand*\xbar[1]{\hbox{\vbox{\hrule height 0.5pt \kern0.5ex\hbox{\kern-0.1em\ensuremath{#1}\kern-0.1em}}}}

\renewcommand{\ge}{\geqslant}
\renewcommand{\le}{\leqslant}

\begin{document}

\title{Means of iterates}
\author{Szymon Draga}
\address{Diebold Nixdorf, Gawronów 4, PL-40-527 Katowice}
\email{szymon.draga@gmail.com}
\author{Janusz Morawiec}
\address{Instytut Matematyki, Uniwersytet Śl\c{a}ski, Bankowa 14, PL-40-007 Katowice, Poland}
\email{morawiec@math.us.edu.pl}
\date{}

\begin{abstract}
We determine continuous bijections $f$, acting on a real interval into itself, whose $k$-fold iterate is the quasi-arithmetic mean of all its subsequent iterates from $f^0$ up to $f^n$ (where $0\le k\le n$). Namely, we prove that if at most one of the numbers $k,n$ is odd, then such functions consist of at most three affine pieces.
\end{abstract}
	
\maketitle
	

\section{Introduction}
Consider integers $n\ge 2$, $0\le k\le n$ and a~non-trivial interval $I\subset\R$. The following problem arises when studying polynomial-like iterative equations, means of functions or convergence of means of iterates: Determine all continuous functions $F\colon I\to I$ with the $k$-fold iterate $F^k$ being a mean of all its subsequent iterates up to $F^n$; more precisely, find all self-mappings $F\in C(I)$ such that 
\begin{equation}\label{E}
F^k(x)=M(x,F(x),F^2(x),\ldots,F^n(x))
\end{equation}
for every $x\in I$, where $M\colon I^{n+1}\to I$ is a mean and $C(I)$ stands for the family of all continuous real functions acting on $I$.
	
In this paper we concentrate on equation \eqref{E} in the case where $M$ is a~quasi-arithmetic mean, i.e., with $M$ being of the form
\begin{equation}\label{M}
M(x_0,...,x_n)=\varphi^{-1}\left(\frac{1}{n+1}\sum_{i=0}^n\varphi(x_i)\right)
\end{equation}
for all $x_0,\ldots,x_n\in I$, where $\varphi$ is a~continuous bijection form $I$ onto an interval $J\subset\R$. Observe that it is enough to consider equation \eqref{E} with $M$ being the arithmetic mean which follows from the following remark.
	
\begin{remark}\label{rem11}
Assume that $J\subset\R$ is an interval, $\varphi\colon I\to J$ is a~bijection onto $J$ and $F\colon I\to I$. Then $F$ satisfies \eqref{E} with $M$ given by \eqref{M} on $I$ if and only if $f=\varphi\circ F\circ\varphi^{-1}$ satisfies 
\begin{equation}\label{e}
f^k=\frac{1}{n+1}\sum_{i=0}^nf^i
\end{equation} 
on $J$. Moreover, $f$ is a~(continuous) bijection if and only if $F$ is a~(continuous) bijection.
\end{remark}

It is easy to prove that any self-mapping satisfying \eqref{e} is injective (see, e.g. \cite[Lemma 2.1]{DM2016}). Therefore, any continuous surjective solution to equation \eqref{e} is automatically continuous bijection; in particular it is strictly monotone.
	
The paper is organized as follows. In Section \ref{polandrec} we collect basic properties of polynomial-like iterative equations and preliminary information on linear homogeneous recurrence relations. In Section \ref{case0n} we solve equation \eqref{e} in the cases where $k=0$ or $k=n$. The case where $0<k<n$ is treated in Section \ref{cases}. The final section contains examples and remarks on the considered equation as well as some further problems.


\section{Polynomial-like iterative equations and recurrence relations}\label{polandrec}
Equation \eqref{e} is a very special case of an iterative equation belonging to an interesting and widely studied class of functional equations, called polynomial-like iterative equations, of the form
\begin{equation}\label{eq}
\sum_{i=0}^na_if^i(x)=0,
\end{equation}
where $f^k$ stands for the $k$-fold iterate of a self-mapping unknown function $f$ (defined on an interval $I\subset\R$) and  $a_0,a_1,\ldots,a_n$ are given real coefficients. It turns out that continuous solutions to \eqref{eq} deeply depend on the roots of its characteristic equation
\begin{equation}\label{che}
\sum_{i=0}^na_ir^i=0
\end{equation}
which is obtained by assuming that $f$ has the form $f(x)=rx$.
	
In general, it is very difficult to find all continuous functions satisfying equation \eqref{eq}. These difficulties follow from the non-linearity of the operator $f\mapsto f^n$. Up to now the complete solution is known only in the case where $n=2$ and $I=\R$ \cite{N1974}. The problem still remains open even for $n=3$~\cite{M1989}. A~partial solution in the case where $n=3$ and $I=\R$ was obtained in \cite{ZG2014} and some results in the case where $n\ge 3$ and $I\neq\R$ can be found in \cite{MR1983,RL1990,Ja1996,ZB2000,ZXZ2013,LZ2013}. 
	
It is known \cite{MZ1998} that if a~polynomial $\sum_{i=0}^mb_ir^i$ divides a~polynomial $\sum_{i=0}^na_ir^i$ and a~function $f$ satisfies $\sum_{i=0}^mb_if^i(x)=0$ then it satisfies also \eqref{eq}. One of methods for finding solutions to \eqref{eq} is based on a~reverse reasoning, i.e., if a~continuous function $f$ satisfies equation \eqref{eq}, then we want to find a~divisor of the polynomial $\sum_{i=0}^na_ir^i$ such that $f$ is a~solution to the corresponding equation of lower order. First such results on the whole real line were obtained in \cite{MZ1997} in the case where all roots of the characteristic equation are real and satisfy some special conditions. Further research in this direction was obtained in \cite{YZ2004,ZZ2010,DM2016,D2016} and some of them will be crucial tools in this paper.
	
Equation \eqref{che} can be considered as the characteristic equation of the recurrence relation
\begin{equation}\label{rec}
\sum_{i=0}^n a_{i}x_{m+i}=0
\end{equation}
which might be obtained by choosing $x_0\in I$ arbitrarily and putting $x_m=f(x_{m-1})$ for every $m\in\N$. It is easy to see that if $a_0\neq0$ and a~function $f$ satisfies \eqref{eq}, then it is injective (see, e.g., \cite[Lemma 2.1.]{DM2016}). This observation implies that every continuous solution to \eqref{eq} is strictly monotone. It means that the sequence $(x_m)_{m\in\N_0}$ given by $x_0\in I$ and $x_m=f(x_{m-1})$ for $m\in\N$ is either monotone (in the case of increasing $f$) or anti-monotone (in the case of decreasing $f$). By \textit{anti-monotone} we mean that the expression $(-1)^m(x_m-x_{m-1})$ does not change its sign when $m$ runs through $\N_0$. In the case where $f$ is bijective we can consider the dual equation
\begin{equation}\label{eqd}
\sum_{i=0}^na_if^{n-i}(x)=0.
\end{equation}
Putting $f^{-n}(x)$ in place of $x$ we see that $f$ satisfies \eqref{eq} if and only if $f^{-1}$ satisfies \eqref{eqd}. We can also extend the sequence $(x_m)_{m\in\N_0}$ to the whole $\Z$ by setting $x_{-m}=f^{-1}(x_{-m+1})$ for $m\in\N$. Then relation \eqref{rec} is satisfied for $m\in\Z$.
	
For the theory of linear recurrence relations we refer the reader, for instance, to \cite[\S 3.2]{J1996}. We shall recall only the most significant theorem in this matter. In order to do this and simplify the writing we introduce the following notation: For a given polynomial $c_nr^n+\ldots+c_1r+c_0$ we denote by ${\mathcal R}(c_n,\ldots,c_0)$ the collection $\{(r_1,k_1),\ldots,(r_p,k_p)\}$ of all  pairs of its pairwise distinct (complex) roots $r_1,\ldots,r_p$ and their multiplicities $k_1,\ldots,k_p$, respectively. Here and throughout the paper by a~polynomial we mean a polynomial with real coefficients.
\begin{theorem}\label{thm31}
Assume that
\begin{align*}
{\mathcal R}(a_n,\ldots,a_0)&=\{(\lambda_1,l_1),\ldots,(\lambda_p,l_p),\\
&\hspace{4ex}(\mu_1,m_1),(\xbar{\mu_1},m_1),\ldots,(\mu_q,m_q),(\xbar{\mu_q},m_q)\}. 
\end{align*}
Then a real-valued sequence $(x_j)_{j\in\mathbb N_0}$ satisfies $(\ref{rec})$ if and only if it is given by
$$
x_j=\sum_{k=1}^pA_k(j)\lambda_k^j+\sum_{k=1}^q\big(B_k(j)\cos j\phi_k+C_k(j)\sin j\phi_k\big)|\mu_k|^j
$$
for every $j\in\N_0$, where $A_k$ is a polynomial whose degree equals at most $l_k-1$ for $k=1,\ldots,p$ and $B_k,C_k$ are polynomials whose degrees equal at most $m_k-1$, with $\phi_k$ being an argument of $\mu_k$, for $k=1,\ldots,q$.
\end{theorem}


\section{Cases $k=0$ and $k=n$}\label{case0n}
In the case where $k=0$ equation \eqref{e} takes the form
\begin{equation}\label{e0}
x=\frac{1}{n}\sum_{i=1}^nf^i(x).
\end{equation}
Its characteristic equation is of the form 
\begin{equation}\label{che0}
\sum_{i=1}^nr^i-n=0
\end{equation} 
which, after multiplication by $r-1$, can be written as $r^{n+1}-(n+1)r+n=0$. Therefore, by the previous remarks, if a~function $f$ satisfies \eqref{e0}, then it also satisfies the equation 
$f^{n+1}(x)-(n+1)f(x)+nx=0$. Now, applying \cite[Theorem 5.7]{DM2016}, we obtain the following result.
\begin{theorem}\label{thm41}
If $f\in C(I)$ satisfies \eqref{e0}, then:
\begin{enumerate}[\rm(i)]
\item $f(x)=x$ in the case where  $n$ is an odd number or in the case where $n$ is an even number and $I\neq\R$;
\item $f(x)=x$ or $f(x)=r_0x+c$, where $c$ is a constant and $r_0$ stands for the negative root of equation \eqref{che0}, in the case where $n$ is an even number and $I=\R$.
\end{enumerate}	
\end{theorem}
	
Now we shall consider the case where $k=n$. Then equation \eqref{e} takes the form
\begin{equation}\label{en}
f^n(x)=\frac{1}{n}\sum_{i=0}^{n-1}f^i(x)
\end{equation}
and its characteristic equation is of the form 
\begin{equation}\label{chen}
nr^n-\sum_{i=0}^{n-1}r^i=0.
\end{equation} 
	
We will need the following lemma which is an elaboration of \cite[Theorem 9]{N1974}; the crucial point is that an unknown function acts on a~subinterval of the real line.
\begin{lemma}\label{lem51}
Assume that $f\in C(I)$ satisfies
\begin{equation}\label{eqrho}
f^2(x)-(1+\rho)f(x)+\rho x=0.
\end{equation}
If either
\begin{enumerate}
\item[\rm (i)] $\rho\in(-1,0)$. If 
\end{enumerate}
or
\begin{enumerate}
\item[\rm (ii)] $\rho\in(-\infty,-1)$ and $f$ is a surjection,
\end{enumerate}
then $f(x)=x$ or $f(x)=\rho x+c$, where $c$ is a constant.
\end{lemma}

\begin{proof}
(i) In the case where $f$ is increasing we can apply \cite[Theorem 4.1 and Remark 4.4]{DM2016} to conclude that $f(x)=x$. Therefore, assume $f$ is decreasing.
		
Fix $x\in I$ and put $x_0=x$ and $x_m=f(x_{m-1})$ for $m\in\N$. It is easy to see that the sequence $(x_m)_{m\in\N_0}$ satisfies the relation
$$
x_{m+2}-(1+\rho)x_{m+1}+\rho x_m=0\quad\mbox{ for }m\in\N_0. 
$$
By Theorem \ref{thm31} we have $x_m=A+B\rho^m$ for some constants $A$ and $B$ (depending on $x$). Consequently, there exists a~finite limit $\lim_{m\to\infty}f^m(x)$.
		
Fix $x,y\in I$. Since $f$ is decreasing, the sequence $(f^m(x)-f^m(y))_{m\in\mathbb N_0}$ is anti-monotone. Moreover, it is convergent, so it must converge to zero. It means that the limit $\lim_{m\to\infty}f^m(x)$ does not depend on $x$. 
		
We can rewrite equation (\ref{eqrho}) as $f^2(x)-\rho f(x)=f(x)-\rho x$. Putting $f(x)$ in place of $x$ we obtain
$$
f^3(x)-\rho f^2(x)=f^2(x)-\rho f(x)=f(x)-\rho x
$$	
and, by a simple induction, we obtain
$$
f^{m+1}(x)-\rho f^m(x)=f(x)-\rho x\quad\mbox{ for all }m\in\N\mbox{ and }x\in I.
$$
Passing with $m$ to $\infty$ gives
$$
f(x)=\rho x+(1-\rho)\lim_{m\to\infty}f^m(x)\quad\mbox{ for every }x\in I. 
$$
Setting $c=(1-\rho)\lim_{m\to\infty}f^m(x)$ ends the proof of assertion (i).

(ii) It is enough to apply assertion (i) to the dual equation to \eqref{eqrho}.
\end{proof}
	
\begin{theorem}\label{thm52}
If $f\in C(I)$ satisfies \eqref{en}, then:
\begin{enumerate}[\rm(i)]
\item $f(x)=x$ in the case where $n$ is an odd number;
\item $f(x)=x$ or $f(x)=r_0x+c$, where $c$ is a constant and $r_0$ stands for the negative root of equation \eqref{chen}, in the case where $n$ is an even number.
\end{enumerate}
\end{theorem}
	
\begin{proof}
Multiplying \eqref{chen} by $r-1$ gives
\begin{equation}\label{chenn}
nr^{n+1}-(n+1)r^n+1=0.
\end{equation}
Clearly, equations \eqref{chen} and \eqref{chenn} have the same roots, wherein multiplicity of root $1$ is greater by $1$ in the second equation. Define a~function $g\colon\R\to\R$ by
$$
g(r)=nr^{n+1}-(n+1)r^n+1.
$$
Since $g(1)=g^\prime(1)=0$ and $g^{\prime\prime}(1)\neq0$, equation \eqref{chenn} has a~double root $1$.
		
In the case of odd $n$, $g$ has the only extremum at $1$, so it is the only real root of $g$. In the case of even $n$, $g$ has extrema at $0$ and $1$, so it also has a negative root $r_0$. Since $g(0)g(-1)<0$, we have $r_0>-1$ and since $g^\prime(r_0)>0$, it follows that $r_0$ is a single root. 
		
We shall show that if $n$ is even and if $z$ is a non-real root of \eqref{chenn}, then $|z|>-r_0$. Suppose, for an indirect proof, the contrary. If $|z|=-r_0$, then we would have
$$
1=|n(-z)^{n+1}+(n+1)z^n|\le n|z|^{n+1}+(n+1)|z|^n=-nr_0^{n+1}+(n+1)r_0^n=1,
$$
which means that $n(-z)^{n+1}$ and $(n+1)z^n$ are linearly dependent over $\R$, and consequently $z\in\R$; a contradiction. If $|z|<-r_0$, then we would have
$$
1=|n(-z)^{n+1}+(n+1)z^n|\le n|z|^{n+1}+(n+1)|z|^n<-nr_0^{n+1}+(n+1)r_0^n=1,
$$
which also is a contradiction.
		
Now by \cite[Theorem 3.6]{DM2016} equation \eqref{en} is equivalent to the equation $f(x)-x=0$ in the case of odd $n$, which proves assertion {\rm (i)}, and to the equation $f^2(x)-(r_0+1)f(x)+r_0x=0$ in the case of even $n$. Now assertion {\rm (ii)} follows from  assertion (i) of Lemma \ref{lem51}.
\end{proof}


\section{Case $0<k<n$}\label{cases}
In order to determine continuous solutions to equation \eqref{e} in the case where $0<k<n$ we need information on roots of its characteristic equation which is of the form 
\begin{equation}\label{c}
(n+1)r^k=\sum_{i=0}^n r^i.
\end{equation}
We start with a general observation on complex roots of equation \eqref{c}.
	
\begin{lemma}\label{lm1}
All complex roots of equation \eqref{c} are in modulus less than $2n+1$. 
\end{lemma}
	
\begin{proof}
Assume that $z\in\mathbb C$ is a root of equation \eqref{c}. Without loss of generality we can assume that $|z|>1$. Then
$$
|z|^n=\left|(n+1)z^k-\sum_{i=0}^{n-1}z^i\right|\le (n+1)|z|^k+\sum_{i=0}^{n-1}|z|^i,
$$
and hence 
$$|z|\le(n+1)|z|^{k-n+1}+\sum_{i=0}^{n-1}|z|^{i-n+1}<(n+1)+n=2n+1.$$
This completes the proof.
\end{proof} 
	
To obtain more precise information on real roots of equation \eqref{c} define functions $g,G\colon\R\to\R$ by
$$
g(r)=r^{n-k+1}-(k+1)r+k,\quad G(r)=r^{n+1}-(n+1)r^k(r-1)-1.
$$
It is easy to see that 
\begin{equation}\label{fg10}
G(1)=g(1)=0,\quad g(0)>0\quad\mbox{and}\quad G(0)<0.
\end{equation}
Moreover, the set of all real roots of equation \eqref{c} coincides with the set of all solutions to the equation 
\begin{equation}\label{cc}
G(r)=0
\end{equation}
and
$$ 
G^\prime(r)=(n+1)r^{k-1}g(r). 
$$
	
Note also that in the case where $n-k$ is odd the function $g$ is strictly convex with the global minimum at the point 
$$
r_{\min}=\left(\frac{k+1}{n-k+1}\right)^{1/(n-k)}
$$
and, by \eqref{fg10}, we have
\begin{equation}\label{gmin}
g(r_{\min})<0\quad\mbox{if and only if}\quad r_{\min}\neq 1.
\end{equation}
However, in the case where $n-k$ is even the function $g|_{(0,+\infty)}$ is strictly convex with the global minimum at the point $r_{\min}$ and \eqref{gmin} holds, whereas the function $g|_{(-\infty,0)}$ is strictly concave with the global maximum at the point $r_{\max}=-r_{\min}$ and, by \eqref{fg10}, we have
\begin{equation}\label{gmax}
g(r_{\max})>0.
\end{equation}
	
Described properties of functions $g$ and $G$ will play a key role in the next lemma which will allow us to locate real roots of equation \eqref{c}.
	
\begin{lemma}\label{lm2}
Assume that $0<k<n$.
\begin{enumerate}[\rm(i)]
\item If $k$ is an odd number and $n=2k$, then equation \eqref{c} has one real root: $r_1=1$. 
\item If $k$ is an odd number and $n$ is an even number with $n<2k$, then equation \eqref{c} has two real roots: $r_1=1$ and $r_2\in(1,2n+1)$.
\item If $k$ is an odd number and $n$ is an even number  with $n>2k$, then equation \eqref{c} has two real roots: $r_1=1$ and $r_2\in(0,1)$.
\item If $k$ and $n$ are odd numbers with $n<2k$, then equation \eqref{c} has three real roots: $r_1=1$, $r_2\in(1,2n+1)$ and $r_3\in(-2n-1,-1)$. 
\item If $k$ and $n$ are odd numbers with $n>2k$, then equation \eqref{c} has three real roots: $r_1=1$, $r_2\in(0,1)$ and $r_3\in(-2n-1,-1)$. 
\item If $k$ is an even number and $n$ is an odd number with $n<2k$, then equation \eqref{c} has three real roots: $r_1=1$, $r_2\in(1,2n+1)$ and $r_3\in(-1,0)$. 
\item If $k$ is an even number and $n$ is an odd number with $n>2k$, then equation \eqref{c} has three real roots: $r_1=1$, $r_2\in(0,1)$ and $r_3\in(-1,0)$.
\item If $k$ is an even number  and $n=2k$, then equation \eqref{c} has three real roots: $r_1=1$, $r_2\in(-1,0)$ and $r_3\in(-2n-1,-1)$.
\item If $k$ and $n$ are even numbers with $n<2k$, then equation \eqref{c} has four real roots: $r_1=1$, $r_2\in(1,2n+1)$, $r_3\in(-1,0)$ and $r_4\in(-2n-1,-1)$. 
\item If $k$ and $n$ are even numbers with $n>2k$, then equation \eqref{c} has four real roots: $r_1=1$, $r_2\in(0,1)$, $r_3\in(-1,0)$ and $r_4\in(-2n-1,-1)$.
\end{enumerate}	
Moreover, in all the above cases $r_1,r_2,r_3,r_4$ are single roots, except the case $n=2k$ in which $r_1$ is a double root.
\end{lemma}
	
\begin{proof}
We first observe that if $k$ is odd, then for every $r\in\R\setminus\{0\}$ we have
$$ 
G^\prime(r)>0\iff g(r)>0. 
$$
		
(i) The assumption $n=2k$ implies $r_{\min}=1$. It yields $g(r)>g(r_{\min})=0$ for every $r\neq r_{\min}$. Consequently, the function $G$ is strictly increasing, which jointly with \eqref{fg10} shows that $r_{\min}$ is the unique real root of equation \eqref{c}. 
\smallskip
		
(ii) The assumption $n<2k$ implies $r_{\min}>1$. We conclude that there exists exactly one $t_0\in\R\setminus\{1\}$ such that $g(t_0)=0$ and, moreover, $t_0>r_{\min}$. Consequently, the function $G$ has exactly one local maximum at the point $r_1=1$ and exactly one local minimum at the point $t_0$. This jointly with \eqref{fg10} shows that equation \eqref{c} has two real roots: $r_1=1$ and $r_2\in (t_0,\infty)$. By Lemma \ref{lm1} we have $r_2<2n+1$.
\smallskip
		
(iii) The reasoning is similar as in (ii).
\smallskip
		
(iv) The assumption $n<2k$ implies $r_{\min}>1$ and $r_{\max}<-1$. We conclude that there exists exactly one $t_0\in(0,1)\cup(1,\infty)$ such that $g(t_0)=0$ and, moreover, $t_0\in(r_{\min},\infty)$. Further, there exists exactly one $u_0\in(-\infty,0)$ such that $g(u_0)=0$ and, moreover, $u_0\in(-\infty,r_{\max})$. Consequently, the function $G$ has exactly one local maximum at the point $r_1=1$ and exactly two local minimums at points $t_0$ and $u_0$. This jointly with \eqref{fg10} shows that equation \eqref{c} has three real roots: $r_1=1$, $r_2\in(t_0,\infty)$ and $r_3\in(-\infty,u_0)$. By Lemma \ref{lm1} we have $r_2<2n+1$ and $r_3>-2n-1$.
\smallskip
		
(v) The reasoning is similar as in (iv).
\medskip
		
Now observe that if $k$ is even, then for every $r\in(0,\infty)$ we have
$$ 
G^\prime(r)>0\iff g(r)>0 
$$
and for every $r\in(-\infty,0)$ we have
$$ 
G^\prime(r)>0\iff g(r)<0. 
$$
		
(vi) The assumption $n<2k$ implies $r_{\min}>1$. We conclude that there exists exactly one $t_0\in\R\setminus\{1\}$ such that $g(t_0)=0$ and, moreover, $t_0>r_{\min}$. Consequently, the function $G$ has exactly one local maximum at the point $r_1=1$ and exactly two local minimums at points $0$ and $t_0$. This jointly with \eqref{fg10} shows that equation \eqref{c} has three real roots: $r_1=1$, $r_2\in (t_0,\infty)$ and $r_3\in(-\infty,0)$. By Lemma \ref{lm1} we have $r_2<2n+1$ and sine
\begin{equation}\label{fmax}
G(-1)>0
\end{equation}
in this case, we have $r_3>-1$.
\smallskip
		
(vii) The reasoning is similar as in (vi).
\smallskip
		
(viii) The assumption $n=2k$ implies $r_{\min}=1$ and $r_{\max}=-1$. We conclude that $g(r)>g(r_{\min})=0$ for every $r\in(0,r_{\min})\cup(r_{\min},\infty)$. Furthermore, there exists exactly one $y_0\in(-\infty,0)$ such that $g(y_0)=0$ and, moreover, $y_0\in(-\infty,r_{\max})$. Consequently, the function $G$ has exactly one local maximum at the point $y_0$ and exactly one local minimum at point $0$. This jointly with \eqref{fg10} and \eqref{fmax} shows that equation \eqref{c} has three real roots: $r_1=1$, $r_2\in(-1,0)$ and $r_3\in(-\infty,-1)$. By Lemma \ref{lm1} we have $r_3>-2n-1$.
\smallskip
		
(ix) The assumption $n<2k$ implies $r_{\min}>1$ and $r_{\max}<-1$. We conclude that there exists exactly one $t_0\in(0,1)\cup(1,\infty)$ such that $g(t_0)=0$ and, moreover, $t_0\in(r_{\min},\infty)$. Furthermore, there exists exactly one $u_0\in(-\infty,0)$ such that $g(u_0)=0$ and, moreover, $u_0\in(-\infty,r_{\max})$. Consequently, the function $G$ has exactly two local maximums at points $r_1=1$ and $u_0$ and exactly two local minimums at points $0$ and $t_0$. This jointly with \eqref{fg10} and \eqref{fmax} shows that equation \eqref{c} has four real roots: $r_1=1$, $r_2\in(t_0,\infty)$, $r_3\in(-1,0)$ and $r_4\in(-\infty,-1)$. By Lemma~\ref{lm1} we have $r_2<2n+1$ and $r_4>-2n-1$.
\smallskip
		
(x) The reasoning is similar as in (ix).

To prove the moreover part note that equality $g^\prime(r)=0$ can hold only for $r\in\{r_{\min},r_{\max}\}$. Thus $G^{\prime\prime}(r_i)=(n+1)r_i^{k-1}g'(r_i)\neq 0$ for every $i\in\{2,3,4\}$ and for $i=1$ in the case where $n\neq 2k$. If $n=2k$, then $G^{\prime\prime}(r_1)=0$ and $G^{\prime\prime\prime}(r_1)=(2k+1)(k+1)k\neq 0$.
\end{proof}
	
Now we want to obtain certain information on the location of non-real roots of equation \eqref{c}. It will be more convenient for us to consider \eqref{cc} instead of \eqref{c}.
	
\begin{lemma}\label{lm3}
Assume that $0<k<n$ and at least one of the numbers $k$ and $n$ is odd. Then for a non-real root $z$ and a real root $r_0$ of equation \eqref{c} we have $|z|\neq|r_0|$.
\end{lemma}
	
\begin{proof}
Let $r_0\in(0,\infty)$; in this case parity of $k$ or $n$ does not play any role. Suppose, towards a contradiction, $|z|=r_0$. Then
$$
(n+1)r_0^k=(n+1)|z|^k=\left|\sum_{i=0}^n z^i\right|<\sum_{i=0}^n |z|^i=\sum_{i=0}^n r_0^i,
$$
which is impossible.
		
Now assume $r_0\in(-\infty,0)$ and as before, suppose that $|z|=-r_0$. By Lemma \ref{lm2} we need to consider only the case where $n$ is odd. 

In the case where both $k$ and $n$ are odd, we have
\begin{align*}
r_0^{n+1} &=|z|^{n+1}=\left|(n+1)z^{k+1}-(n+1)z^{k}+1\right|\\
&<(n+1)|z|^{k+1}+(n+1)|z|^{k}+1=(n+1)r_0^{k+1}-(n+1)r_0^{k}+1;
\end{align*}
a contradiction. Similarly, in the case where $k$ is even and $n$ is odd, we have
\begin{align*}
1&=\left|z^{n+1}-(n+1)z^{k+1}+(n+1)z^k\right|<|z|^{n+1}+(n+1)|z|^{k+1}+(n+1)|z|^k\\
&=r_0^{n+1}-(n+1)r_0^{k+1}+(n+1)r_0^k;
\end{align*}
a contradiction.
\end{proof}
	
The last lemma we need is basically Theorem 8 and Theorem 10(iii) from \cite{N1974}; the only difference is that in our lemma the unknown function acts on a~subinterval of the real line.
\begin{lemma}\label{lem73}
Let $\rho\in(0,\infty)$ and assume that a surjection $f\in C(I)$ satisfies \eqref{eqrho}.
\begin{enumerate}[\rm (i)]
\item If $\rho=1$, then $f(x)=x+c$, where $c$ is a constant.
\item If $\rho\in(0,1)\cup(1,\infty)$, then 
\begin{equation}\label{ab}
f(x)=\begin{cases}
\rho(x-a)+a&\hbox{for }x\leq a,\\
x&\hbox{for }x\in (a,b),\\
\rho(x-b)+b&\hbox{for }x\geq b,\\
\end{cases}
\end{equation}
where $a,b\in{\rm cl}I$ with $a\leq b$.
\end{enumerate}
\end{lemma}	
\begin{proof}
Since $f(x)=\frac{1}{1+\rho}\left(f^2(x)+\rho x\right)$, it follows that $f$ is strictly increasing.\\
(i) If $\rho=1$, then the monotonicity of $f$ and a simple induction applied to \eqref{eqrho} imply
$$
0<\frac{f^m(x)-f^m(y)}{x-y}=1+m\left(\frac{f(x)-f(y)}{x-y}-1\right)
$$
for all $m\in\mathbb Z$ and $x,y\in I$ with $x\neq y$. Therefore, passing with $m$ in turn to $\infty$ and $-\infty$ we conclude that $\frac{f(x)-f(y)}{x-y}= 1$ for all $x,y\in I$.\\
(ii) Assume that $\rho\in(0,1)$; the case when $\rho\in(1,\infty)$ can be proved similarly, or else can be deduced from the dual case.\\
Applying induction to \eqref{eqrho} with $\rho\neq 1$ we obtain
$$
f^m(x)=\frac{1}{\rho-1}(\rho x-f(x))+\frac{\rho^m}{\rho-1}(f(x)-x)
$$
for all $m\in\mathbb Z$ and $x\in I$. If $f(x)\neq x$ for some $x\in I$, then passing with $m$ to $\infty$ we get
$$
\lim_{m\to\infty}f^m(x)=\frac{1}{\rho-1}(\rho x-f(x)),
$$
which means that $f(x)=\rho(x-a)+a$ with some $a\in {\rm cl}I$. In consequence \eqref{ab} holds.
\end{proof}

Note that if $a=\inf I$ and $b=\sup I$, then \eqref{ab} represents the identity solution, while if $a=b$, then \eqref{ab} represents an affine solution.

\begin{theorem}\label{thm71}
Assume $0<k<n$. If $f\in C(I)$ is a surjection satisfying \eqref{e}, then:
\begin{enumerate}[\rm(i)]
\item $f(x)=x+c$, where $c$ is a~constant, in the case where $k$ is an odd number and $n=2k$;
\item 
\begin{equation}\label{picewise-linear}
f(x)=\begin{cases}
r_0(x-a)+a&\hbox{for }x\leq a,\\
x&\hbox{for }x\in (a,b),\\
r_0(x-b)+b&\hbox{for }x\geq b,\\
\end{cases}
\end{equation}
where $a,b\in{\rm cl}I$ with $a\leq b$ and $r_0$ stands for the different from $1$ positive root of equation \eqref{c}, in the case where $k$ is an odd number and $n$ is an even number with $n\neq 2k$;
\item 
$f(x)=r_0x+c$, where $r_0$ stands for the negative root of equation \eqref{c}, or $f$ is of form \eqref{picewise-linear}, where $a,b\in{\rm cl}I$ with $a\leq b$ and $r_0$ stands for the different from $1$ positive root of equation \eqref{c}, in the case where $k$ is an even number and $n$ is an odd number or both $k$ and $n$ are odd numbers.
\end{enumerate}
\end{theorem}

\begin{proof} 
The idea of the proof in each of the considered cases is the same and run in the following way. First we determine all real roots of equation \eqref{c} with their multiplicities and localize all its non-real roots. This step is done in Lemmas \ref{lm2} and \ref{lm3}. Next, according to \cite[Theorem 3.6, Corollary 3.7 and Theorem 4.3]{DM2016}, we reduce equation \ref{e} to a simpler one, i.e. to an equation which have the same surjective and continuous solutions. Finally, we solve the simpler equation applying Lemmas \ref{lem51} and \ref{lem73}.
\end{proof}
	
\begin{remark}
In the case where $I=\R$  the surjectivity assumption in Theorem \ref{thm71} is satisfied automatically. However, this assumption is not essential when considering the dual equation is not necessary.
\end{remark}


\section{Examples, remarks and problems}\label{example}
	
In this section we give some examples showing how our main results works.
	
\begin{corollary}\label{cor81}
Assume that $p\neq0$ and $I\subset(0,\infty)$. Then $f\in C(I)$ satisfies
$$
f^n(x)=\left(\frac{x^p+[f(x)]^p+\ldots+[f^{n-1}(x)]^p}{n}\right)^{1/p}
$$
if and only if
\begin{enumerate}[\rm(i)]
\item $f(x)=x$ in the case where $n$ is an odd number;
\item $f(x)=x$ or $f(x)=(r_0x^p+c)^{1/p}$, where $c$ is a constant such that $f(I)\subset I$ and $r_0$ stands for the negative root of equation $(\ref{chen})$, in the case where $n$ is an even number.
\end{enumerate}
\end{corollary}
	
\begin{proof}
It is enough to use Theorem \ref{thm52} and Remark \ref{rem11} with $\varphi$ of the form $\varphi(x)=x^{1/p}$. 
\end{proof}
	
\begin{corollary}\label{cor82}
Assume that $I\subset(0,\infty)$. Then $f\in C(I)$ satisfies
$$
f^n(x)=\sqrt[n]{xf(x)\ldots f^{n-1}(x)}
$$
if and only if
\begin{enumerate}[\rm(i)]
\item $f(x)=x$ in the case where $n$ is an odd number;
\item $f(x)=x$ or $f(x)=cx^{r_0}$, where $c$ is a constant such that $f(I)\subset I$ and $r_0$ stands for the negative root of equation \eqref{chen}, in the case where $n$ is an even number.
\end{enumerate}
\end{corollary}
	
\begin{proof}
It is enough to use Theorem \ref{thm52} and Remark \ref{rem11} with $\varphi=\exp$.
\end{proof}
	
\begin{remark}\label{rem83}
In the case where $I=\R$ equation \eqref{en} is the dual equation to \eqref{e0}. However, as we have shown in Corollary \ref{cor81}, there exist non-surjective solutions to \eqref{en}.
\end{remark}
	
\begin{corollary}\label{cor84}
Assume that $p\neq0$ and $I\subset(0,\infty)$. Then $f\in C(I)$ satisfies
\begin{equation}\label{eqpowboros}
\frac{[f(x)]^p+[f^2(x)]^p+\ldots+[f^n(x)]^p}{n}=x^p
\end{equation}
if and only if $f(x)=x$. 
\end{corollary}
	
\begin{proof}
It is enough to use Theorem \ref{thm41} and Remark \ref{rem11} with $\varphi$ of the form $\varphi(x)=x^{1/p}$. The root $r_0$ has no influence for solutions to \eqref{eqpowboros}, because the interval $\{x^{1/p}\colon x\in I\}$ cannot be equal to the whole real line.
\end{proof}
	
\begin{corollary}\label{cor85}
Assume that $I\subset(0,\infty)$. Then $f\in C(I)$ satisfies	
$$
f(x)f^2(x)\ldots f^n(x)=x^n
$$
if and only if
\begin{enumerate}[\rm(i)]
\item $f(x)=x$ in the case where $n$ is an odd number or $I\neq(0,\infty)$;
\item $f(x)=x$ or $f(x)=cx^{r_0}$, where $c$ is a constant such that $f(I)\subset I$ and $r_0$ stands for the negative root of equation \eqref{che0}, in the case where $n$ is an even number and $I=(0,\infty)$.
\end{enumerate}
\end{corollary}
	
\begin{proof}
It is enough to use Theorem \ref{thm41} and Remark \ref{rem11} with $\varphi=\exp$.
\end{proof}

\begin{corollary}\label{cor86}
Assume that $p\neq0$, $m\in\N$ and $I\subset(0,\infty)$. Then $f\in C(I)$ satisfies
\begin{equation}\label{eqpowboros1}
\frac{[f^m(x)]^{p}+[f^{2m}(x)]^{p}+\ldots+[f^{nm}(x)]^p}{n}=x^p
\end{equation}
if and only if 
\begin{enumerate}[\rm(i)]
\item $f(x)=x$ in the case where $m$ is an odd number or $I$ is neither open nor closed;
\item $f(x)=x$ or 
\begin{equation}\label{f0}
f(x)=\begin{cases}
f_0(x)&\hbox{ for every }x\le a,\\
f_0^{-1}(x)&\hbox{ for every }x<a,\\
\end{cases}
\end{equation}
where $a$ is an arbitrary interior point of $I$ and $f_0$ is an arbitrary continuous and strictly decreasing function defined on $I\cap(0,a]$ such that  $\lim_{x\to\inf I}f_0(x)=\sup I$ and $f_0(a)=a$, in the case where $m$ is an even number and $I$ is open or closed. 
\end{enumerate}
\end{corollary}

\begin{proof}
From Corollary \ref{cor84} we conclude that 
\begin{equation}\label{fm}
f^m(x)=x.
\end{equation}
Now it is enough to apply \cite[Theorem 15.3, Theorem 15.2 and Lemma 15.2]{K1968} jointly with the fact that there is no continuous strictly decreasing function satisfying \eqref{fm} defined on an interval which is neither open nor closed.
\end{proof}

\begin{corollary}\label{cor87}
Assume that $m\in\N$ and $I$ is bounded. Then a bijection $f\in C(I)$ satisfies	
\begin{multline*}
\frac{\exp\left(f^m(x)\right) +\exp\left(f^{m+1}(x)\right)+\ldots+\exp\left(f^{m+4n-2}(x)\right)}{4n-1} \\ =\exp\left(f^{m+2n-1}(x)\right)
\end{multline*}
if and only if
\begin{enumerate}[\rm(i)]
\item $f(x)=x$ in the case where $m$ is an odd number or $I$ is neither open nor closed;
\item $f(x)=x$ or $f$ is of the form \eqref{f0}, where $a$ is an arbitrary interior point of $I$ and $f_0$ is an arbitrary continuous and strictly decreasing function defined on $I\cap(-\infty,a]$ such that  $\lim_{x\to\inf I}f_0(x)=\sup I$ and $f_0(a)=a$, in the case where $m$ is an even number and $I$ is open or closed.
\end{enumerate}
\end{corollary}

\begin{proof}
Applying assertion (i) of Theorem \ref{thm71} and Remark \ref{rem11} with $\varphi=\log$ we conclude that \eqref{fm} holds.
Now we argue as in Corollary \ref{cor86}.
\end{proof}

We finish this paper with two problems motivated by the main results.
\begin{problem}
Can we omit the sujectivity assumption on $f$ in Theorem \ref{thm71}?
\end{problem}
\begin{problem}
Determine all (bijections) $f\in C(I)$ satisfying \eqref{e} in the case where both the numbers $k$ and $n$ are even.
\end{problem}
	
\subsection*{Acknowledgement}
Research of the second-named author was supported by the University of Silesia Mathematics Department (Iterative Functional Equations and Real Analysis program).


\begin{thebibliography}{99}
\bibitem{D2016} Draga, S.: A note on the polynomial-like iterative equations order, Comment. Math. 56 (2016), 243--249.
\bibitem{DM2016} Draga, S., Morawiec, J.: Reducing the polynomial-like iterative equations order and a generalized Zolt\'an Boros problem, Aequationes Math. 90 (2016), 935--950.
\bibitem{Ja1996} Jarczyk, W.: On an equation of linear iteration, Aequationes Math. 51 (1996), 303--310.
\bibitem{J1996} Jerri, A.J.: Linear Difference Equations with Discrete Transform Methods, Springer 1996.
\bibitem{LZ2013} Li, L., Zhang, W.: Continuously decreasing solutions for polynomial-like iterative equations, Sci. China Math. 56 (2013), 1051--1058.
\bibitem{K1968} Kuczma, M: Functional Equations in a Single Variable. Monografie Matematyczne, vol. 46. PWN, Warsaw (1968).
\bibitem{M1989} Matkowski, J.: Remark done during The Twenty-sixth International Symposium on Functional Equations, Aequationes Math. 37 (1989), 119.
\bibitem{MZ1997} Matkowski, J., Zhang, W.: Method of characteristic for functional equations in polynomial form, Acta Math. Sinica 13 (1997),  421--432.
\bibitem{MZ1998} Matkowski, J., Zhang, W.: Characteristic analysis for a polynomial-like iterative equation, Chinese Sci. Bull. 43 (1998), 192--196.
\bibitem{MR1983}  Mukherjea, A., Ratti, J.S.: On a functional equation involving iterates of a bijection on the unit interval, Nonlinear Anal. 7 (1983), 899--908; Nonlinear Anal. 31 (1998), 459--464.
\bibitem{N1974} Nabeya, S.: On the functional equation $f(p+qx+rf(x))=a+bx+cf(x)$, Aequationes Math. 11 (1974), 199--211.
\bibitem{YZ2004}  Yang, D., Zhang, W.: Characteristic solutions of polynomial-like iterative equations, Aequationes Math. 67 (2004), 80--105.
\bibitem{RL1990}  Ratti, J.S., Lin, Y.F.: A functional equation involving $f$ and $f^{-1}$, Colloq. Math. 60/61 (1990), 519--523.
\bibitem{ZB2000} Zhang, W., Baker, A.J.: Continuous solutions of a polynomial-like iterative equation with variable coefficients, Ann. Polon. Math. 73 (2000), 29--36.
\bibitem{ZG2014} Zhang, P., Gong, X.: Continuous solutions of $3$-order iterative equation of linear dependence, Adv. Difference Equ. 2014 (2014), 318.
\bibitem{ZXZ2013} Zhang, W., Xu, B., Zhang, W.: Global solutions for leading coefficient problem of polynomial-like iterative equations, Results. Math. 63 (2013), 79--93.
\bibitem{ZZ2010} Zhang, W., Zhang, W.: On continuous solutions of $n$-th order polynomial-like iterative equations, Publ. Math. Debrecen 76 (2010), 117--134.
\end{thebibliography}
\end{document}